\documentclass[a4paper]{amsart}
\usepackage[utf8]{inputenc}
\usepackage[T1]{fontenc}
\usepackage{lmodern}
\usepackage{amssymb,amsxtra}
\usepackage[all]{xy}
\usepackage{nicefrac,mathtools,enumitem}
\usepackage{mdwlist}
\setlist[enumerate,1]{label=\textup{(\arabic*)}}
\usepackage{lmodern,textcomp}

\usepackage{tikz}
\usetikzlibrary{matrix}
\tikzset{cd/.style=matrix of math nodes,row sep=2em,column sep=2em, text height=1.5ex, text depth=0.5ex}
\tikzset{cdar/.style=->,auto}
\tikzset{overar/.style={draw=white,double=black,double distance=.4pt,very thick}}

\usetikzlibrary{positioning,calc}
\usetikzlibrary{shapes}
\usetikzlibrary{matrix}
\usetikzlibrary{arrows}

\usepackage{microtype}
\usepackage[pdftitle={On the path correspondences of quantum graphs},
  pdfauthor={Sourav Khatua and Sutanu Roy},
  pdfsubject={Mathematics; MSC }
]{hyperref}

\usepackage[lite]{amsrefs}

\renewcommand{\PrintDOI}[1]{\href{http://dx.doi.org/\detokenize{#1}}{doi: \detokenize{#1}}}

\BibSpec{book}{%
  +{}  {\PrintPrimary}                {transition}
  +{,} { \textit}                     {title}
  +{.} { }                            {part}
  +{:} { \textit}                     {subtitle}
  +{,} { \PrintEdition}               {edition}
  +{}  { \PrintEditorsB}              {editor}
  +{,} { \PrintTranslatorsC}          {translator}
  +{,} { \PrintContributions}         {contribution}
  +{,} { }                            {series}
  +{,} { \voltext}                    {volume}
  +{,} { }                            {publisher}
  +{,} { }                            {organization}
  +{,} { }                            {address}
  +{,} { \PrintDateB}                 {date}
  +{,} { }                            {status}
  +{}  { \parenthesize}               {language}
  +{}  { \PrintTranslation}           {translation}
  +{;} { \PrintReprint}               {reprint}
  +{.} { }                            {note}
  +{.} {}                             {transition}
  +{} { \PrintDOI}                   {doi}
  +{} { available at \url}            {eprint}
  +{}  {\SentenceSpace \PrintReviews} {review}
}

\numberwithin{equation}{section}

\theoremstyle{plain}
\newtheorem{theorem}[equation]{Theorem}

\newtheorem{proposition}[equation]{Proposition}

\theoremstyle{definition}
\newtheorem{definition}[equation]{Definition}

\theoremstyle{remark}

\newtheorem{example}[equation]{Example}

\usepackage{dsfont}

\newcommand*{\nb}{\nobreakdash}
\newcommand*{\Star}{$^*$\nb-}

\newcommand*{\C}{\mathbb C}

\newcommand*{\R}{\mathbb R}
\newcommand*{\N}{\mathbb N}

\newcommand*{\QSp}[1][]{(B_{#1},\psi_{#1})}
\newcommand*{\Qgh}{\mathcal{G}}
\newcommand*{\QghTr}[1][]{(B_{#1},\psi_{#1},A_{#1})}
\newcommand*{\QghFull}[1][]{\Qgh_{#1}=(B_{#1},\psi_{#1},A_{#1})}
\newcommand*{\Eind}[1][]{\epsilon_{#1}}
\newcommand*{\Ecor}[1][]{\mathcal{E}_{#1}}

\newcommand*{\Cont}{\textup C}

\newcommand*{\Id}{\textup{id}}

\newcommand*{\flip}{\chi}

\newcommand*{\Cst}{\textup C^*}

\newcommand*{\Hils}[1][H]{\mathcal{#1}}

\newcommand*{\defeq}{\mathrel{\vcentcolon=}}

\allowdisplaybreaks

\begin{document}
\title{On the path correspondences of quantum graphs}

\author{Sourav Khatua}
\email{24d0796@iitb.ac.in}
\address{Department of Mathematics\\
 Indian Institute of Technology Bombay\\
 Powai, Mumbai, 400076, Maharashtra\\
 India}

\author{Sutanu Roy}
\email{sutanu@iitb.ac.in}
\address{Department of Mathematics\\
 Indian Institute of Technology Bombay\\
 Powai, Mumbai, 400076, Maharashtra\\
 India}

\begin{abstract}
We introduce notions of path indicators and path correspondences for finite quantum graphs, study their basic properties, and compute them explicitly for classical graphs, trivial graphs, and complete quantum graphs.
\end{abstract}

\thanks{The second author was partially supported by Seed Grant (No. RD/0525-IRCCSH0-022) from the IRCC at IIT Bombay. He also thanks Professor Debashish Goswami for stimulating discussions.}

\maketitle

\section{Introduction}
 \label{sec:intro}
 A finite quantum graph is a noncommutative generalisation of a finite graph in which
the vertex set is replaced by a finite-dimensional \(\Cst\)\nb-algebra \(B\)
and the edge set is encoded by one of the two equivalent formulations: a quantum adjacency operator~\cite{MRV18}, or an operator bimodule over
\(B'\)~\cites{Wea12,D2024a}, each generalising the adjacency matrix of a
classical graph. Quantum graphs arise as confusability graphs of quantum channels~\cite{DSW2013a} and underlie the study of the graph isomorphism~\cites{BCEHPSW2020a,MRV2019a}. The subject has developed rapidly in recent years, with progress on classification results in low dimensions~\cites{Gromada2021a,Matsuda2022}, operator algebras associated with quantum graphs and their properties~\cites{BEVW2022a,BHINW2023a,HIN26}, connectivity~\cite{CGW25}, and quantum Cayley graphs~\cite{W2024a}.
 
 \medskip
 
Paths are among the most informative structures carried by a graph.
In particular, the concatenation of edges into paths also provides the fundamental data for
building the graph \(\Cst\)\nb-algebras. It is therefore natural, in passing to
quantum graphs, to ask how edges concatenate into quantum paths and what
structure the resulting path spaces carry, a question raised in the
introduction of~\cite{BHINW2023a}. The aim of this paper is to develop this
theory for paths of finite length, in the spirit of~\cite{BHINW2023a}, by
introducing and studying path correspondences for finite quantum graphs.

\medskip

We begin by fixing notation and recalling the necessary definitions and
results in Section~\ref{sec:prelim}. We shall work with the quantum adjacency
operator formulation of quantum graphs~\cite{MRV18}. The main results are presented in
Section~\ref{sec:path-corr}. For a given finite quantum graph~\(\Qgh\), we
define the sequence~\(\{\Eind[\Qgh]^{k}\}_{k=0}^{\infty}\) of path indicators,
construct the sequence~\(\{\Ecor[\Qgh]^{k}\}_{k=0}^{\infty}\) of path
correspondences over the vertex algebra~\(B\) of~\(\Qgh\), and study some of
their basic properties. In particular, we show that the concatenation maps
\(U_{r,s}\colon \Ecor[\Qgh]^{r}\otimes_{Z(B)}\Ecor[\Qgh]^{s} \to
\Ecor[\Qgh]^{r+s}\) are \(Z(B)\)\nb-balanced and turn \(\{\Ecor[\Qgh]^{k}\}_{k=0}^{\infty}\) into a subproduct system~\cite{ShS09} over~\(Z(B)\). Finally, in Section~\ref{sec:examples}, we compute the edge indicators and edge correspondences for classical graphs, complete quantum graphs, and trivial quantum graphs.

\section{Preliminaries}
 \label{sec:prelim}
\subsection{Hilbert modules and $\Cst$-correspondences} A \(\Cst\)\nb-correspondence over a \(\Cst\)\nb-algebra \(B\) is a right Hilbert \(B\)\nb-module \(\Hils\) along with a \Star{}homomorphism \(\pi_{\Hils}\colon B\to\mathcal{L}(\Hils)\), where \(\mathcal{L}(\Hils)\) denotes the right \(B\)\nb-linear adjointable operators on \(\Hils\). This gives a left \(B\)\nb-action on \(\Hils\), that commutes with the right \(B\)\nb-action on \(\Hils\), by multiplication: \(b\cdot \xi:=\pi_{\Hils}(b)\xi\) for all \(b\in B\) and \(\xi\in \Hils\). Furthermore, \(\Hils\) is said to be \emph{faithful} if \(\pi_{\Hils}\) is faithful, \emph{nondegenerate} or \emph{essential}, if \(B\cdot \Hils\) is dense in \(\Hils\).

Let \(\Hils_{1}\) and~\(\Hils_{2}\) be \(\Cst\)\nb-correspondences over~\(B\). Suppose \(\pi_{\Hils_{i}}\colon B\to\Hils[L](\Hils_{i})\) are \Star{}homomorphisms inducing the left \(B\)\nb-action on~\(\Hils_{i}\), for \(i=1,2\). The algebraic tensor product \(\Hils_{1}\odot\Hils_{2}\) carries a natural \(B\)\nb-valued semi-inner product defined by \(\langle x_{1}\otimes y_{1}, x_{2}\otimes y_{2}\rangle_{B}:=\langle y_{1},\pi_{\Hils_{2}}(\langle x_{1},x_{2}\rangle_{B})y_{2}\rangle_{B}\), for all  \(x_{1},x_{2}\in \Hils_{1}\) and \(y_{1},y_{2}\in\Hils_{2}\). Then we obtain a right Hilbert \(B\)\nb-module~\(\Hils_{1}\otimes_{B}\Hils_{2}\) from \(\Hils_{1}\odot\Hils_{2}\) by separation and completion. The kernel of the quotient map~\(\Hils_{1}\odot\Hils_{2}\to\Hils_{1}\otimes_{B}\Hils_{2}\) is the linear span of \(\{xb\otimes y-x\otimes\pi_{\Hils_{2}}(b)y\mid x\in\Hils_{1}, y\in\Hils_{2}, b\in B\}\). Moreover, there is a left \(B\)\nb-action, commuting with the right \(B\)\nb-action, given by \(b\cdot (x\otimes y):=\pi_{\Hils_{1}}(b)x\otimes y\), for all~\(x_{1}\in\Hils_{1}\), \(y\in\Hils_{2}\) and \(b\in B\). Thus, the interior tensor product~\(\Hils_{1}\otimes_{B}\Hils_{2}\) canonically becomes a \(\Cst\)\nb-correspondence over~\(B\). 

\begin{example}
 \label{ex:ext-tens-fuse}
Let \(B\) be a finite dimensional \(\Cst\)\nb-algebra and let \(\psi\colon B\to\C\) be a faithful state. Then \(B\) acts on \(B\otimes B\) from the right by \((x\otimes y)\cdot b=x\otimes yb\). We equip~\(B\otimes B\) with a \(B\)\nb-valued inner product \(\langle \xi_{1}, \xi_{2}\rangle_{B}:=(\psi\otimes\Id_{B})(\xi_{1}^{*}\xi_{2})\), for all~\(\xi_{1},\xi_{2}\in B\otimes B\). Faithfulness of~\(\psi\) and finite dimensionality of \(B\) make the separation and completion redundant; hence \(B \otimes B\) is a right Hilbert \(B\)\nb-module.  Furthermore, \(B\otimes B\) is a \(\Cst\)\nb-correspondence over \(B\) with respect to the left action of~\(B\) on the first tensor factor of \(B\otimes B\), that is, \(b\cdot (x_{1}\otimes x_{2}):=ax_{1}\otimes x_{2}\), for all~\(b,x_{1},x_{2}\in B\). 

More generally, for nonnegative integer \(n\), we may view~\(B^{\otimes (n+1)}\) as \(\Cst\)\nb-correpondence over~\(B\). We define a \(B\)\nb-valued inner product on \(B^{\otimes (n+1)}\) by \(\langle \xi_{1}, \xi_{2}\rangle_{B}:=(\psi^{\otimes n}\otimes\Id_{B})(\xi_{1}^{*}\xi_{2})\) for all \(\xi_{1},\xi_{2}\in B^{\otimes (n+1)}\). The left and right actions of \(B\) on \(B^{\otimes (n+1)}\) are given by \(b_{1}\cdot (x_{1}\otimes\cdots \otimes x_{n+1})\cdot b_{2}:=b_{1}x_{1}\otimes \cdots \otimes x_{n+1}b_{2}\) for all \(b_{1},b_{2},x_{1},\cdots , x_{n+1}\in B\). In particular, \(B^{\otimes (n+1)}\) is faithful and essential \(B\)\nb-Hilbert module.
\end{example}

\subsection{Finite quantum spaces} Let \(B\) be a  dimensional \(\Cst\)\nb-algebra~\(B\) and let \(\psi\colon B\to\C\) be a faithful state. Then we obtain a Hilbert space \(L^{2}(B)\defeq L^{2}(B,\psi)\) by equipping \(B\) with the inner product \(\langle x,y\rangle=\psi(x^{*}y)\). Faithfulness of \(\psi\) ensures \(L^{2}(B)=B\). The multiplication map~\(m\colon B\otimes B\to B\) is viewed as a linear operator \(L^{2}(B)\otimes L^{2}(B)\to L^{2}(B)\), and \(m^{*}\colon L^{2}(B)\to L^{2}(B)\otimes L^{2}(B)\) is the Hilbert space adjoint~\(m\). For a given \(\delta\in\R\), a state~\(\psi\colon B\to\C\) is said to be a \emph{\(\delta\)\nb-form} if \(mm^{*}=\delta^{2}\Id_{B}\).

This gives us the standard prototype for \emph{finite quantum spaces} as pairs \(\QSp\), consisting of a finite dimensional \(\Cst\)\nb-algebra \(B\) and a faithful \(\delta\)\nb-form \(\psi\colon B\to\C\). 

Since \(B\) is finite dimensional, there exists \(N\in\N\) and \(n_{1},\cdots ,n_{N}\in \N\) such that \(B=\oplus_{a=1}^{N}M_{n_{a}}(\C)\).  Let \(\{e_{kl}^{a}\mid  k,l=1,\cdots ,n_{a}, a=1,\cdots , N\}\) be a set of standard matrix units. Without of loss of generality, we assume that the density matrix with respect to the usual trace \(Q=\oplus_{a=1}^{N}Q^{(a)}\) is diagonal, with the \(k\)\nb-th entry of \(Q^{(a)}\) being \(\psi(e_{kk}^{a})\). Furthermore, \(\psi\) is a \(\delta\)\nb-form if and only if \(Q^{(a)}\) is invertible and \(\sum_{k=1}^{n_{a}}\psi(e_{kk}^{a})^{-1}=\delta^{2}\) for all~\(a=1,\cdots ,N\). The analytic extension of the modular automorphism group for \(\psi\) is defined by \(\sigma_{z}^{\psi}(b)\defeq Q^{iz}b Q^{-iz}\) for all~\(b\in B\), \(z\in\C\). Recall the adapted matrix units~\(\{f_{kl}^{a}\mid k,l=1,\cdots ,n_{a}, a=1,\cdots , N\}\) for the quantum space \(\QSp\) from~\cite{BHINW2023a}. They are defined by 
\begin{equation}
 \label{eq:adapted-mat-unit}
 f_{kl}^{a}:=\frac{1}{\sqrt{\psi(e_{kk}^{a})\psi(e_{ll}^{a})}}e_{kl}^{a},
 \qquad
 \text{\(1\leq k,l\leq n_{a}, 1\leq a\leq N\).}
\end{equation}
Furthermore, they satisfy the following relations:
\begin{align}
 \label{eq:adapted-mat-rel}
 m^{*}(f_{kl}^{a}):=\sum_{r=1}^{n_{a}}f^{a}_{kr}\otimes f^{a}_{rl}, 
 &\qquad
 f^{a}_{kl}f^{b}_{rs}:= \delta_{a,b}\delta_{l,r}\frac{1}{\psi(e^{a}_{ll})} f_{ks}^{a}.
\end{align}
We also note that \(1_{B}=\sum_{a=1}^{N}\sum_{k=1}^{n_{a}}\psi(e_{kk}^{a})f_{kk}^{a}\) and 
\(\sigma_{-i}^{\psi}(f_{kl}^{a})=\frac{\psi(e_{kk}^{a})}{\psi(e_{ll}^{a})}f_{kl}^{a}\) for all~\(k,l=1,\cdots , n_{a}\), \(N=1,\cdots ,N\).
 
 \section{Quantum path indicators and path correspondences}
  \label{sec:path-corr}
  Let \(\QSp\) be a finite quantum space with a \(\delta\)\nb-form \(\psi\). A linear operator \(A\colon L^{2}(B)\to L^{2}(B)\) is said to be a \emph{quantum adjacency matrix} if 
 \begin{equation}
  \label{eq:q-adj-matrix}
  m(A\otimes A)m^{*}=\delta^{2}A .
 \end{equation}
 A \emph{finite directed quantum graph} \(\Qgh\) is a triple \(\QghTr\) consisting of a finite quantum space \(\QSp\) and a \emph{quantum adjacency matrix} \(A\colon L^{2}(B)\to L^{2}(B)\).
 
Suppose \(\QghFull\) is a finite directed quantum graph. Define~\(T\colon B\otimes B\to B\) by 
  \begin{equation}
   \label{eq:def-T}
    T\defeq m\circ\flip\circ(\sigma^{\psi}_{-i}\otimes\Id_{B}),
    \qquad \text{where \(\flip\) is the tensor flip.}
  \end{equation}   
Then \(T^*=(\sigma_{-i}^{\psi}\otimes\Id_{B})\circ\chi\circ m^*\) is the Hilbert space adjoint of the linear operator \(T\colon L^{2}(B)\otimes L^{2}(B)\to L^{2}(B)\).  
  
  \begin{definition}
  \label{def:k-path-new}
  Suppose~\(k\) is a positive integer. The \emph{k\nb-length quantum path indicator} \(\Eind[\Qgh]^{k}
  \in B^{\otimes (k+1)}\) is defined recursively as follows: 
  \begin{equation}
   \label{eq:k-path-ind}
   \Eind[\Qgh]^{0}=1_{B},
   \quad
   \text{and}
   \quad
   \Eind[\Qgh]^{k}=\frac{1}{\delta^{2}}(\Id_{B^{\otimes (k-1)}} \otimes (\Id_{B}\otimes A)T^{*}) \Eind[\Qgh]^{k-1} 
   \in B^{\otimes (k+1)} 
   \quad
   \text{for \(k\geq 1\)}.
  \end{equation}

  The \emph{k\nb-length quantum path space} is defined by 
   \begin{equation}
    \label{eq:k-path-corr}
    \Ecor[\Qgh]^{k}
    \defeq \text{span}\{(x\otimes 1_{B})\Eind[\Qgh]^{k} (1_{B}\otimes y)\mid x, y\in B^{\otimes k}\}
    \subseteq B^{\otimes (k+1)}.
   \end{equation}
\end{definition}
We recall the edge indicator~\(\Eind[\Qgh]\) of~\(\mathcal{G}\) is defined by \(\Eind[\Qgh]\defeq \frac{1}{\delta^{2}}(\Id_{B}\otimes A)m^{*}(1_{B})\in B\otimes B\) and the edge correspondence \(\Ecor[\Qgh]\defeq B \cdot \Eind[\Qgh] \cdot B\subseteq B\otimes B\) from \cite{BHINW2023a}. 
  In particular, we get \(\Eind[\Qgh]^{1}=\Eind[\Qgh]\). Indeed,  
  \begin{align*} 
     \frac{1}{\delta^{2}}(\Id_{B}\otimes A)T^*(1_{B})
    &=\frac{1}{\delta^{2}}(\sigma_{-i}^{\psi} \otimes A)\flip\left(\sum_{a=1}^{N}\sum_{k=1}^{n_{a}}\psi(e_{kk}^{a}) m^*(f_{kk}^{a})\right)\\
    &=\frac{1}{\delta^{2}}(\Id_{B} \otimes A) \left(\sum_{a=1}^{N}\sum_{k,l=1}^{n_{a}}\psi(e_{kk}^{a}) \sigma_{-i}^{\psi}(f_{lk}^{a})\otimes f_{kl}^{a} \right)\\
    &=\frac{1}{\delta^{2}}(\Id_{B} \otimes A) \left(\sum_{a=1}^{N}\sum_{l,k=1}^{n_{a}} \psi(e_{ll}^{a}) f_{lk}^{a} \otimes f_{kl}^{a} \right)\\
    &=\frac{1}{\delta^{2}}(\Id_{B} \otimes A)m^{*}(1_{B}).
  \end{align*}

 Subsequently, \(\Ecor[\Qgh]^{1}=\Ecor[\Qgh]\) is the edge correspondence, which is a \(\Cst\)\nb-subcorrespondence of \(B\otimes B\)  over~\(B\). Furthermore, \(\Ecor[\Qgh]^{k}\) is a \(\Cst\)\nb-subcorrespondence of~\(B^{\otimes (k+1)}\) over~\(B\) for every positive integer \(k\).

\subsection{Properties of path indicators and correspondences}
In this section, we record basic properties of quantum path indicators and quantum path spaces. 
  
   \begin{proposition}
  \label{prop:target}
  Suppose \(k\) is a positive integer. Then \(A^{k}(x)=\delta^{2^k}(\psi^{\otimes k}\otimes\Id_{B})(x\cdot \Eind[\Qgh]^{k})\) for all \(x\in B\).
 \end{proposition}
 \begin{proof}
  The statement for~\(k=1\) is proved in~\cite{BHINW2023a}*{Proposition 2.3 (1)}. Suppose the statement holds true for \(k=n\). For 
  \(x,y\in B\), we have 
  \begin{align*}
   & \langle \delta^{2^{n+1}}(\psi^{\otimes (n+1)}\otimes 1_{B})(x\cdot\Eind[\Qgh]^{n+1}),y\rangle_{\psi}\\
   &=\delta^{2^{n+1}} \langle \Eind[\Qgh]^{n+1}, x^{*}\otimes 1_{B}^{\otimes n} \otimes y\rangle_{\psi^{\otimes (n+2)}} \\
   &= \delta^{2^{n}}\langle (\Id_{B^{\otimes n}} \otimes (\Id_{B}\otimes A)T^{*}) \Eind[\Qgh]^{n}, x^{*}\otimes 1_{B}^{\otimes n} \otimes y\rangle_{\psi^{\otimes (n+2)}} \\
    &= \delta^{2^{n}}\langle \Eind[\Qgh]^{n}, x^{*}\otimes 1_{B}^{\otimes (n-1)} \otimes T(1_{B}\otimes A^{*}(y))\rangle_{\psi^{\otimes (n+1)}}\\
    &= \delta^{2^{n}}\langle \Eind[\Qgh]^{n}, x^{*}\otimes 1_{B}^{\otimes (n-1)} \otimes A^{*}(y))\rangle_{\psi^{\otimes (n+1)}}\\
    &=\langle \delta^{2^{n}} (\psi^{\otimes n}\otimes 1_{B})(x\cdot \Eind[\Qgh]^{n}), A^{*}(y) \rangle_{\psi}
   =\langle A^{n}(x),A^*(y)\rangle_{\psi}
   =\langle A^{n+1}(x),y\rangle_{\psi}.\qedhere
  \end{align*}
 \end{proof}

  For any pair of positive integers~\((r,s)\), define 
  \(\iota_{r,s}\colon B^{\otimes (r+1)}\otimes B^{\otimes (s+1)}\to B^{\otimes (r+s+1)}\) by 
  \begin{equation}
   \label{eq:def-iota}
   \iota_{r,s}(x\otimes y)\defeq (\Id_{B^{\otimes r}}\otimes T \otimes \Id_{B^{\otimes s}})(x\otimes y),
   \quad\text{for all \(x\in B^{\otimes (r+1)}\), \(y\in B^{\otimes (s+1)}\).}
  \end{equation}
In the following, we find an alternative description of quantum path indicators. 

\begin{proposition}
 \label{prop:equiv-Eind}
 Suppose \(k\) is a non\nb-negative integer. Then \(\Eind[\Qgh]^{k+1}=\iota_{1,k}(\Eind[\Qgh]\otimes\Eind[\Qgh]^{k})\).
\end{proposition}
\begin{proof}
  The statement is trivial for \(k=0\).  For the remaining cases, we proceed by induction. Let the matrix coefficients \(\{A^{rsb}_{kla}\mid k,l=1\cdots , n_{a}, r,s=1,\cdots ,n_{b}, a,b=1,\cdots ,N\}\) of \(A\) with respect to the adapted matrix units \(\{f_{kl}^{a}\mid kl,l=1,\cdots , n_{a}, a=1,\cdots , N\}\) be given by 
  \[
     A(f_{kl}^{a})\defeq\sum_{b=1}^{N}\sum_{r,s=1}^{n_{b}}A_{kla}^{rsb}f_{rs}^{b}.
  \]
  Consequently, 
  \begin{align*}
    \Eind[\Qgh]
    =\frac{1}{\delta^{2}}(\Id_{B}\otimes A)m^{*}(1_{B})
    &=\frac{1}{\delta^{2}}\sum_{a=1}^{N}\sum_{k,l=1}^{n_{a}}\psi(e^{a}_{kk})f_{kl}^{a}\otimes A(f_{lk}^{a})\\
    &=\frac{1}{\delta^{2}}\sum_{a,b=1}^{N}\sum_{k,l=1}^{n_{a}}\sum_{r,s=1}^{n_{b}}\psi(e^{a}_{kk})A_{lka}^{rsb} f_{kl}^{a}\otimes f_{rs}^{b}.
  \end{align*}
  Using it, we compute
  \begin{align*}
   &\iota_{1,1}(\Eind[\Qgh]\otimes\Eind[\Qgh])\\
   &=\frac{1}{\delta^{4}} \iota_{1,1} 
   \left(
   \sum_{a,b=1}^{N}\sum_{k,l=1}^{n_{a}}\sum_{r,s=1}^{n_{b}}\psi(e^{a}_{kk})A_{lka}^{rsb} f_{kl}^{a}\otimes f_{rs}^{b} 
   \otimes 
   \sum_{c=1}^{N}\sum_{x,y=1}^{n_{c}} \psi(e_{xx}^{c})f_{xy}^{c}\otimes A(f_{yx}^c)
   \right)\\
   &=\frac{1}{\delta^{4}} \sum_{a,b,c=1}^{N}\sum_{k,l=1}^{n_{a}}\sum_{r,s=1}^{n_{b}}\sum_{x,y=1}^{n_{c}} 
    \psi(e^{a}_{kk})\psi(e_{xx}^{c})A_{lka}^{rsb} f_{kl}^{a}\otimes f_{xy}^{c}\sigma_{-i}^{\psi}(f_{rs}^{b}) \otimes A(f_{yx}^c)\\
   &=\frac{1}{\delta^{4}} \sum_{a,b,c=1}^{N}\sum_{k,l=1}^{n_{a}}\sum_{r,s=1}^{n_{b}}\sum_{x,y=1}^{n_{c}} 
     \frac{\psi(e^{a}_{kk})\psi(e_{xx}^{c})\psi(e_{rr}^{b})}{\psi(e_{ss}^{b})} A_{lka}^{rsb} f_{kl}^{a}\otimes f_{xy}^{c}f_{rs}^{b} \otimes A(f_{yx}^c)\\
   &=\frac{1}{\delta^{4}} \sum_{a,b=1}^{N}\sum_{k,l=1}^{n_{a}}\sum_{r,s,x=1}^{n_{b}} 
     \frac{\psi(e^{a}_{kk})\psi(e_{xx}^{b})}{\psi(e_{ss}^{b})} A_{lka}^{rsb} f_{kl}^{a}\otimes f_{xs}^{b} \otimes A(f_{rx}^b)\\
   &=\frac{1}{\delta^{4}} \sum_{a,b=1}^{N}\sum_{k,l=1}^{n_{a}}\sum_{r,s,x=1}^{n_{b}} \psi(e_{kk}^{a}) A^{rsb}_{lka} f_{kl}^{a}\otimes \sigma_{-i}^{\psi}(f_{xs}^{b}) \otimes A(f_{rx}^b)\\
    &=\frac{1}{\delta^{4}} \sum_{a,b=1}^{N}\sum_{k,l=1}^{n_{a}}\sum_{r,s=1}^{n_{b}} \psi(e_{kk}^{a}) A^{rsb}_{lka} f_{kl}^{a}\otimes (\Id_{B}\otimes A)T^*(f_{rs}^{b})\\
    &=\frac{1}{\delta^{4}} \sum_{a=1}^{N}\sum_{k,l=1}^{n_{a}} \psi(e_{kk}^{a}) f_{kl}^{a}\otimes (\Id_{B}\otimes A)T^{*}(A(f_{lk}^{a}))\\
    &=\frac{1}{\delta^{4}}  
    \left(\Id_{B}\otimes(\Id_{B}\otimes A)T^{*}\right)\left((\Id_{B}\otimes A)m^{*}\left(\sum_{a=1}^{N}\sum_{k=1}^{n_{a}} \psi(e_{kk}^{a})f_{kk}^{a}\right)\right)\\
    &=\frac{1}{\delta^{2}} (\Id_{B}\otimes (\Id_{B}\otimes A)T^{*})\Eind[\Qgh]
    =\Eind[\Qgh]^{2}.
  \end{align*}
  Hence, the statement holds true for \(k=1\). Suppose the statement holds true for \(k=n\). We compute
  \begin{align*}
   & \iota_{1,n+1}(\Eind[\Qgh]\otimes\Eind[\Qgh]^{n+1})\\ 
   &= (\Id_{B}\otimes T\otimes \Id_{B^{\otimes (n+1)}}) 
        \left(\Eind[\Qgh]\otimes \frac{1}{\delta^{2}}(\Id_{B^{\otimes n}}\otimes(\Id_{B}\otimes A)T^{*})\Eind[\Qgh]^{n}\right)\\
   &=\frac{1}{\delta^{2}}(\Id_{B^{\otimes (n+1)}}\otimes (\Id_{B}\otimes A)T^{*})
        \left (\Id_{B}\otimes T\otimes \Id_{B^{\otimes n}})(\Eind[\Qgh]\otimes\Eind[\Qgh]^{n})\right)\\
   &=\frac{1}{\delta^{2}}(\Id_{B^{\otimes (n+1)}}\otimes (\Id_{B}\otimes A)T^{*})\Eind[\Qgh]^{n+1}=\Eind[\Qgh]^{n+2};       
  \end{align*}
  the first and the fourth equalities use~\eqref{eq:k-path-ind}, the second equality is trivial, and the third equality uses \eqref{eq:def-iota} and the the induction hypothesis. 
\end{proof}
The above description of quantum path indicators becomes particularly useful because it provides a convenient framework for defining the composition of quantum path indicators and the associated quantum path correspondences. We conclude this section with a proof of the following result concerning the composition of quantum path indicators and quantum path correspondences.
 \begin{theorem}
  \label{the:prop-p-ind-p-mod}
  Suppose \(r,s\) are positive integers. Then
  \begin{enumerate}
   \item\label{eq:p-ind-asso} \(\Eind[\Qgh]^{r+s}
   =\iota_{r,s} (\Eind[\Qgh]^{r}\otimes \Eind[\Qgh]^{s})\);
   \item\label{eq:p-corr-asso} \(\Ecor[\Qgh]^{r+s}=\iota_{r,s}(\Ecor[\Qgh]^{r}\otimes\Ecor[\Qgh]^{s})\).
   \end{enumerate} 
 \end{theorem}
 \begin{proof}
 Firstly, we prove by induction, that for every positive integer~\(r\)
   \begin{equation}
    \label{eq:edge-ind-stitch}
      \Eind[\Qgh]^{r}=(\Id_{B}\otimes T^{\otimes (r-1)}\otimes\Id_{B})(\Eind[\Qgh]^{\otimes r}) 
      \in B^{\otimes (r+1)}.
   \end{equation}
   For~\(r=1\), the equality is trivial. Assume for \(r=r'\) the equation~\eqref{eq:edge-ind-stitch} holds true. 
Using \eqref{prop:equiv-Eind} together with the induction hypothesis in the following computation, we obtain \eqref{eq:edge-ind-stitch} for 
\(r=r'+1\).
   \begin{align*}
     \Eind[\Qgh]^{r'+1}
        =\iota_{1,r'}(\Eind[\Qgh]\otimes\Eind[\Qgh]^{r'+1})
     &=(\Id_{B}\otimes T\otimes \Id_{B^{\otimes r'}})(\Eind[\Qgh] \otimes \Eind[\Qgh]^{r'})\\
     &=(\Id_{B}\otimes T\otimes \Id_{B^{\otimes r'}})\left(\Eind[\Qgh] \otimes (\Id_{B}\otimes T^{\otimes (r'-1)}\otimes\Id_{B})\Eind[\Qgh]^{r'}\right) \\
     &=(\Id_{B}\otimes T\otimes T^{\otimes (r'-1)}\otimes \Id_{B})(\Eind[\Qgh]\otimes\Eind[\Qgh]^{\otimes r'})\\
     &=(\Id_{B}\otimes T^{\otimes r'}\otimes\Id_{B})(\Eind[\Qgh]^{\otimes (r'+1)}).
   \end{align*}
   Then combining~\eqref{prop:equiv-Eind} and~\eqref{eq:def-iota} below we prove the statement~\ref{eq:p-ind-asso}.
   \begin{align*}
     & \Eind[\Qgh]^{r+s}
      =(\Id_{B}\otimes T^{\otimes (r+s-1)}\otimes \Id_{B})(\Eind[\Qgh]^{\otimes (r+s)})\\
     &=(\Id_{B^{\otimes r}}\otimes T \otimes \Id_{B^{\otimes s}})
     \left((\Id_{B}\otimes T^{\otimes (r-1)}\otimes\Id_{B})(\Eind[\Qgh]^{\otimes r})\otimes (\Id_{B}\otimes T^{\otimes (s-1)}\otimes\Id_{B})(\Eind[\Qgh]^{\otimes s})\right)\\
     &=(\Id_{B^{\otimes r}}\otimes T \otimes \Id_{B^{\otimes s}})(\Eind[\Qgh]^{r}\otimes \Eind[\Qgh]^{s})
     =\iota_{r,s}(\Eind[\Qgh]^{r}\otimes \Eind[\Qgh]^{s}).
   \end{align*}
   Next we prove by induction that for every positive integer~\(r\)
   \begin{equation}
    \label{eq:edge-corr-stitch}
      \Ecor[\Qgh]^{r}=(\Id_{B}\otimes T^{\otimes (r-1)}\otimes\Id_{B})(\Ecor[\Qgh]^{\otimes r}) 
      \subseteq B^{\otimes (r+1)}.
   \end{equation}
      For~\(r=1\), the equality is trivial. Assume for \(r=r'\) we the equation~\eqref{eq:edge-corr-stitch} holds true. Using 
      \eqref{eq:k-path-corr} together with the induction hypothesis in the following computation, we obtain \eqref{eq:edge-corr-stitch} 
      for \(r=r'+1\).
      \begin{align*}
       & \Ecor[\Qgh]^{r'+1}\\
       &=(B^{\otimes (r'+1)}\otimes 1_{B})\Eind[\Qgh]^{r'+1}(1_{B}\otimes B^{\otimes (r'+1)})\\
       &=(B^{\otimes (r'+1)}\otimes 1_{B})\left((\Id_{B^{\otimes r'}} \otimes T \otimes \Id_{B})(\Eind[\Qgh]^{r'}\otimes \Eind[\Qgh])\right)(1_{B}\otimes B^{\otimes (r'+1)})\\
       &=(\Id_{B^{\otimes r'}} \otimes T \otimes \Id_{B})\\
       & \left((B^{\otimes r'}\otimes 1_{B}\otimes B\otimes 1_{B})(\Eind[\Qgh]^{r'}\otimes \Eind[\Qgh])  
       (1_{B}\otimes B^{\otimes (r'-1)} \otimes\sigma_{i}^{\psi}(B)\otimes 1_{B} \otimes B)\right)\\
       &= (\Id_{B^{\otimes r'}} \otimes T \otimes \Id_{B}) \left((B^{\otimes r'}\otimes 1_{B})\Eind[\Qgh]^{r'}(1_{B}\otimes B^{\otimes r'})
             \otimes (B\otimes 1_{B})\Eind[\Qgh](1_{B}\otimes B) \right)\\
       &=(\Id_{B^{\otimes r'}} \otimes T \otimes \Id_{B})\left(\Ecor[\Qgh]^{r'}\otimes\Ecor[\Qgh]^{1}\right)\\
       &=(\Id_{B^{\otimes r'}} \otimes T \otimes \Id_{B}) \left((\Id_{B}\otimes T^{\otimes (r'-1)}\otimes\Id_{B^{\otimes 3}})
       (\Ecor[\Qgh]^{\otimes r'}\otimes\Ecor[\Qgh]^{1})\right)\\
       &=(\Id_{B} \otimes T^{\otimes r'} \otimes \Id_{B})(\Ecor[\Qgh]^{\otimes (r'+1)}).
      \end{align*}
      Now combining~\eqref{eq:edge-corr-stitch} and \eqref{eq:def-iota} below we prove the statement~\ref{eq:p-corr-asso}.
      \begin{align*}
        & \Ecor[\Qgh]^{r+s}\\
        &=(\Id_{B}\otimes T^{\otimes (r+s-1)}\otimes\Id_{B})(\Ecor[\Qgh]^{\otimes r+s})\\
        &=(\Id_{B^{\otimes r}}\otimes T \otimes \Id_{B^{\otimes s}})
        \left((\Id_{B}\otimes T^{\otimes (r-1)}\otimes\Id_{B})(\Ecor[\Qgh]^{\otimes r})\otimes (\Id_{B}\otimes T^{\otimes (s-1)}\otimes\Id_{B})(\Ecor[\Qgh]^{\otimes s})\right)\\
     &=(\Id_{B^{\otimes r}}\otimes T \otimes \Id_{B^{\otimes s}})(\Ecor[\Qgh]^{r}\otimes \Ecor[\Qgh]^{s})
     =\iota_{r,s}(\Ecor[\Qgh]^{r}\otimes \Ecor[\Qgh]^{s}).\qedhere
      \end{align*}
 \end{proof}
 
 \subsection{Subproduct system of quantum path correspondences}
  \label{sec:subprod-sys}
  Suppose \(\mathcal{G}=(B,\psi,A)\) is a finite directed quantum graph. Suppose \(p_{1},\cdots ,p_{N}\) are the minimal central projections of \(B\). The center \(Z(B)\) of \(B\) is the \(\Cst\)\nb-subalgebra of \(B\) generated by \(\{p_{a}\mid a=1,\cdots ,N\}\). Let \(E\colon B\to Z(B)\) be the unique \(\psi\)\nb-preserving conditional expectation. On the set of the adapted matrix units~\(\{f_{ij}^{a}\mid \text{\(i,j=1,\cdots , n_{a}\), \(a=1,\cdots , N\)}\}\) it is defined by \(E(f_{ij}^{a}):=\delta_{i,j}\frac{1}{\psi(p_{a})}p_{a}\).
  
The left action of \(Z(B)\) for the embedding \(Z(B)\hookrightarrow B\) on \(B^{\otimes (k+1)}\) is faithful, as \(Z(B)B=B\). Furthermore, the faithful conditional expectation~\(E\) induces a \(Z(B)\)\nb-valued inner product on~\(B^{\otimes (k+1)}\): \(\langle \xi_{1},\xi_{2}\rangle_{Z(B)}:=E(\langle \xi_{1},\xi_{2}\rangle_{B})\) for all~\(\xi_{1},\xi_{2}\in B^{\otimes (k+1)}\). Finaly, the right action of~\(Z(B)\) on \(B^{\otimes (k+1)}\) induced by the embedding \(Z(B)\hookrightarrow B\) makes \(B^{\otimes (k+1)}\) a \(\Cst\)\nb-correspondence over~\(Z(B)\). Subsequently, \(\Ecor[\Qgh]^{k}\) becomes a \(\Cst\)\nb-subcorrespondence over~\(Z(B)\) for every positive integer \(k\).

\begin{proposition}
 \label{prop:bimod-iota}
 For all positive integers~\(r,s\), the maps~\(\iota_{r,s}\) defined by~\eqref{eq:def-iota} descents to adjointable \(Z(B)\)\nb-bimodule maps~\(\iota_{r,s}\colon \Ecor[\Qgh]^{r}\otimes_{Z(B)} \Ecor[\Qgh]^{s}\to \Ecor[\Qgh]^{r+s}\). 
\end{proposition}
\begin{proof}
 Clearly, \(\iota_{r,s}\) is a \(Z(B)\)\nb-bimodule map. Next, we show that \(\iota_{r,s}\) respects the \(Z(B)\)\nb-balancing relations. Now, 
 using~\eqref{eq:def-iota}, \eqref{eq:edge-corr-stitch} and Theorem~\ref{the:prop-p-ind-p-mod}~\ref{eq:p-corr-asso}, we observe that 
 \begin{align*}
    \iota_{r,s}(\Ecor[\Qgh]^{r}\otimes\Ecor[\Qgh]^{s})
   &=(\Id_{B}\otimes T^{\otimes (r+s-1)}\otimes \Id_{B})\left(\Ecor[\Qgh]^{\otimes (r+s)}\right)\\
   &=(\Id_{B}\otimes T^{\otimes (r-1)}\otimes\Id_{B}\otimes T^{\otimes (s-1)}\otimes\Id_{B})\\
   & \left(\Ecor[\Qgh]^{\otimes (r-1)} \otimes (\Id_{B}\otimes T\otimes\Id_{B})(\Ecor[\Qgh]^{1}\otimes\Ecor[\Qgh]^{1}) \otimes \Ecor[\Qgh]^{\otimes (s-1)}\right)\\
   &=(\Id_{B}\otimes T^{\otimes (r-1)}\otimes\Id_{B}\otimes T^{\otimes (s-1)}\otimes\Id_{B})\\
   & \left(\Ecor[\Qgh]^{\otimes (r-1)} \otimes \iota_{1,1}(\Ecor[\Qgh]^{1}\otimes\Ecor[\Qgh]^{1}) \otimes \Ecor[\Qgh]^{\otimes (s-1)}\right) 
\end{align*}
Therefore, it is sufficient to verify that \(\iota_{1,1}\) respects the \(Z(B)\)\nb-balancing relations. Note that \(\sigma_{-i}^{\psi}(p_{a})=p_{a}\) for all~\(a=1,\cdots , N\); hence 
\(\sigma_{-i}^{\psi}(d)=d\) for all~\(d\in Z(B)\). For any 
 \(d\in Z(B)\), \(x_{1},x_{2},y_{1},y_{2}\in B\), we have
 \begin{align*}
   & \iota_{1,1}\left((x_{1}\otimes 1_{B})\Eind[\Qgh] (1\otimes y_{1}d) \otimes (x_{2}\otimes 1_{B})\Eind[\Qgh] (1\otimes y_{2})\right)\\
  &= (x_{1}\otimes x_{2}\otimes 1_{B})\left(\iota_{1,1}(\Eind[\Qgh]\otimes\Eind[\Qgh])\right) (1_{B}\otimes \sigma_{-i}^{\psi}(y_{1}d)\otimes y_{2})\\ 
  &=(x_{1}\otimes  x_{2}\otimes 1_{B})\left(\iota_{1,1}(\Eind[\Qgh]\otimes\Eind[\Qgh])\right) (1_{B}\otimes \sigma^{\psi}_{-i}(y_{1})d\otimes y_{2})\\
  &=(x_{1}\otimes  d x_{2}\otimes 1_{B})\left(\iota_{1,1}(\Eind[\Qgh]\otimes\Eind[\Qgh])\right) (1_{B}\otimes \sigma_{-i}^{\psi}(y_{1})\otimes y_{2})\\  
  &=\iota_{1,1}\left((x_{1}\otimes 1_{B})\Eind[\Qgh] (1\otimes y_{1}) \otimes (dx_{2}\otimes 1_{B})\Eind[\Qgh] (1\otimes y_{2})\right).
 \end{align*}
  Therefore, \(\iota_{1,1}\) descents to a well defined linear map~\(\Ecor[\Qgh]^{1}\otimes_{Z(B)}\Ecor[\Qgh]^{1}\to\Ecor[\Qgh]^{2}\). 
 Once again, to show that~\(\iota_{r,s}\) is adjointable, it is sufficient to show that~\(\iota_{1,1}\) is adjointable. In fact, we verify that 
 \(\iota_{1,1}^{*}(x)=\sum_{a=1}^{N}\psi(p_{a})(\Id_{B}\otimes T^{*}\otimes\Id_{B})((1_{B}\otimes p_{a}\otimes 1_{B})x)\) for all~\(x\in 
 \Ecor[\Qgh]^{2}\). Equivalently, we verify the following: 
 \[
\langle \iota_{1,1}(u\otimes f^{b}_{kl}\otimes f^{c}_{rs}\otimes v), x\otimes f_{ij}^{a}\otimes y\rangle_{Z(B)}
=\langle u\otimes f^{b}_{kl}\otimes f^{c}_{rs}\otimes v, \iota_{1,1}^*(x\otimes f_{ij}^{a}\otimes y)\rangle_{Z(B)}
\] 
for all~\(u\otimes f_{kl}^{b}, f_{rs}^{c}\otimes v\in\Ecor[\Qgh]^{1}\) and \(x\otimes f_{ij}^{a}\otimes y \in \Ecor[\Qgh]^{2}\subseteq B^{\otimes 3}\). 

In the following computations we use~\eqref{eq:adapted-mat-unit} and \eqref{eq:adapted-mat-rel}. We simplify the left hand side 
\begin{align*}
 &  \langle \iota_{1,1}(u\otimes f^{b}_{kl}\otimes f^{c}_{rs}\otimes v), x\otimes f_{ij}^{a}\otimes y\rangle_{Z(B)}\\
 &=\langle u\otimes f_{rs}^{c}\sigma_{-i}^{\psi}(f_{kl}^{b})\otimes v, x\otimes f_{ij}^{a}\otimes y\rangle_{Z(B)}\\
 &=\delta_{b,c}\delta_{k,s}\frac{1}{\psi(e^{b}_{ll})}\langle u\otimes f^{b}_{rl} \otimes v, x\otimes f_{ij}^{a}\otimes y\rangle_{Z(B)}\\
 &=\delta_{b,c}\delta_{k,s}\frac{\psi(u^{*}x)\psi(f_{lr}^{b}f_{ij}^{a})}{\psi(e^{b}_{ll})}E(v^{*}y)\\
 &=\delta_{a,b} \delta_{b,c}\delta_{k,s}\delta_{i,r}\delta_{j,l} \frac{\psi(u^{*}x)}{\psi(e^{a}_{ii})\psi(e^{a}_{ll})}E(v^{*}y).
\end{align*}
 Similarly, we simplify the right hand side
\begin{align*}
 &\langle u\otimes f^{b}_{kl}\otimes f^{c}_{rs}\otimes v, T^*(x\otimes f_{ij}^{a}\otimes y)\rangle_{Z(B)}\\
 &=\sum_{t=1}^{n_{a}}\psi(p_{a})\langle u\otimes f^{b}_{kl}\otimes f^{c}_{rs}\otimes v, x\otimes \sigma_{-i}^{\psi}(f_{tj}^{a})\otimes f_{it}^{a}\otimes y)\rangle_{Z(B)}\\
 &=\sum_{t=1}^{n_{a}}\psi(p_{a})\left\langle  f^{c}_{rs}\otimes v,
 \left\langle u\otimes f_{kl}^{b},x\otimes \frac{\psi(e_{tt}^{a})}{\psi(e^{a}_{jj})}f^{a}_{tj}\right\rangle_{Z(B)}\cdot (f_{it}^{a}\otimes y)\right\rangle_{Z(B)}\\
 &=\sum_{t=1}^{n_{a}}\psi(p_{a}) \left\langle  f^{c}_{rs}\otimes v, 
 \frac{\psi(u^{*}x)\psi(e_{tt}^{a})}{\psi(e_{jj}^{a})}
 E(f_{lk}^{b}f_{tj}^{a})f_{it}^{a}\otimes y\right\rangle_{Z(B)}\\
 &=\delta_{a,b} \frac{\psi(p_{a})\psi(u^{*}x)}{\psi(e_{jj}^{a})}\left\langle  f^{c}_{rs}\otimes v, E(f_{lj}^{a})f_{ik}^{a}\otimes y\right\rangle_{Z(B)}\\
 &=\delta_{a,b} \delta_{j,l} \frac{\psi(u^{*}x)}{\psi(e_{jj}^{a})}
 \left\langle  f^{c}_{rs}\otimes v, p_{a}f_{ik}^{a}\otimes y\right\rangle_{Z(B)}\\
 &=\delta_{a,b} \delta_{j,l} \frac{\psi(u^{*}x)}{\psi(e_{jj}^{a})} \psi(f_{sr}^{c}f_{ik}^{a})E(v^{*}y)\\
 &=\delta_{a,b} \delta_{a,c} \delta_{i,r} \delta_{j,l} \frac{\psi(u^{*}x)\psi(f_{sk}^{a})}{\psi(e_{ll}^{a})\psi(e_{ii}^{a})}E(v^{*}y)\\
 &=\delta_{a,b} \delta_{a,c} \delta_{k,s} \delta_{i,r} \delta_{j,l} \frac{\psi(u^{*}x)}{\psi(e_{ll}^{a})\psi(e_{ii}^{a})}E(v^{*}y).
\end{align*} 
Thus, for all positive integers~\(r,s\) and for all~\(x\in \Ecor[\Qgh]^{r+s}\subseteq B^{\otimes (r+s+1)}\), we have 
\[
     \iota_{r,s}^{*}(x) =\sum_{a=1}^{N}  \psi(p_{a})(\Id_{B^{\otimes r}}\otimes T^{*}\otimes \Id_{B^{\otimes s}})
   (1_{B^{\otimes r}} \otimes p_{a}\otimes 1_{B^{\otimes s}})(x)
  \qedhere
\] 
\end{proof}

\begin{theorem}
 \label{the:subprod-sys-qgraph}
 Suppose~\(\Ecor[\Qgh]^{0}=Z(B)\). For every~\(r\in\N_{0}\), define 
 \(U_{0,r}\colon \Ecor[\Qgh]^{0}\otimes_{Z(B)}\otimes\Ecor[\Qgh]^{r}\to \Ecor[\Qgh]^{r}\) and  \(U_{r,0}\colon\Ecor[\Qgh]^{r}\otimes_{Z(B)} \Ecor[\Qgh]^{0}\to\Ecor[\Qgh]^{r}\) as the left and right actions of~\(Z(B)\) on~\(\Ecor[\Qgh]^{r}\). Let~\(\lambda\defeq\frac{1}{\delta}\sum_{a=1}^{N}\frac{1}{\sqrt{\psi(p_{a})}}p_{a}\in Z(B)\subseteq B\). For all positive integers \(r,s\), define \(U_{r,s}\colon \Ecor[\Qgh]^{r}\otimes_{Z(B)}\Ecor[\Qgh]^{s}\to\Ecor[\Qgh]^{r+s}\) by
 \[
   U_{r,s}(x\otimes y)\defeq \left(1_{B^{\otimes r}}\otimes \lambda \otimes 1_{B^{\otimes s}}\right) \iota_{r,s}
   \left((\Id_{B^{\otimes r}}\otimes (\sigma_{i}^{\psi}\otimes\Id_{B})\otimes\Id_{B^{\otimes s}})(x\otimes y)\right),
 \]
 for all \(x\in \Ecor[\Qgh]^{r}\) and \(y\in \Ecor[\Qgh]^{s}\). Let 
 \(E_{\Qgh}\defeq\{\Ecor[\Qgh]^{r}\}_{r\in\N_{0}}\) and let~\(U\defeq \{U_{r,s}\}_{r,s\in\N_{0}}\). Then \((E_{\Qgh},U)\) is a subproduct system over~\(Z(B)\). 
\end{theorem}
\begin{proof}
 For every~\(r\in\N_{0}\) the maps~\(U_{r,0}\) and \(U_{0,r}\) are the unitaries for the canonical identifications \(\Ecor[\Qgh]^{r}\otimes_{Z(B)} Z(B)\cong\Ecor[\Qgh]^{r}\) and~\(Z(B)\otimes_{Z(B)}\Ecor[\Qgh]^{r}\cong\Ecor[\Qgh]^{r}\), respectively. Observe that~\(\lambda\) is a positive element of~\(Z(B)\). For every positive integers \(r,s\), Proposition~\ref{prop:bimod-iota} shows that \(U_{r,s}\) are adjointable. 

Consequently, for all \(x\in \Ecor[\Qgh]^{r+s}\) we get,
 \begin{align*}
    U_{r,s}U_{r,s}^{*}(x)
   &=\frac{1}{\delta^{2}}\sum_{a=1}^{N} 
   (\Id_{B^{\otimes r}}\otimes p_{a}T(\sigma_{2i}^{\psi}\otimes\Id_{B})T^{*}\otimes \Id_{B^{\otimes s}})
    (1_{B^{\otimes r}} \otimes p_{a}\otimes 1_{B^{\otimes s}})(x)\\
   &=\frac{1}{\delta^{2}}\sum_{a=1}^{N} (1_{B^{\otimes r}} \otimes p_{a}mm^{*}p_{a}\otimes 1_{B^{\otimes s}})(x)=x.
 \end{align*}
 Therefore, \(U_{r,s}\) is coisometry for all~\(r,s\in \N_{0}\). Finally, we need to verify \(U_{r+s,t}\circ (U_{r,s}\otimes\Id_{\Ecor[\Qgh]^{t}})=U_{r,s+t}\circ (\Id_{\Ecor[\Qgh]^{r}}\otimes U_{s,t})\) for every for positive integers~\(r,s,t\). Suppose \(x=\sum_{\alpha}x^{1}_{\alpha}\otimes \cdots \otimes x^{r+1}_{\alpha}\in\Ecor[\Qgh]^{r}\subseteq B^{\otimes (r+1)}\),  \(y=\sum_{\beta}y^{1}_{\beta}\otimes \cdots \otimes y^{s+1}_{\beta} \in\Ecor[\Qgh]^{s}\subseteq B^{\otimes (s+1)}\) and \(z=\sum_{\gamma}z^{1}_{\gamma}\otimes \cdots \otimes z^{n+1}_{\gamma}\in\Ecor[\Qgh]^{t}\subseteq B^{\otimes (t+1)}\). Then the following routine computation completes the proof:
  \begin{align*}
   & U_{r+s,t}\bigl((U_{r,s}\otimes\Id_{\Ecor[\Qgh]^{t}})\bigr)(x\otimes y\otimes z)\\
    &=\frac{1}{\delta^{2}}\sum_{a,b\in S}\sum_{\alpha,\beta,\gamma}x^{1}_{\alpha}\otimes \cdots \otimes p_{a}y^{1}_{\beta}x^{r+1}_{\alpha}\otimes y^{2}_{\beta}\otimes\cdots \otimes p_{b}z^{1}_{\gamma}y_{\beta}^{s+1}\otimes\cdots \otimes z^{k+1}_{\gamma}\\
  &= U_{r,s+t}\bigl((\Id_{\Ecor[\Qgh]^{r}}\otimes U_{s,t})\bigr) (x\otimes y\otimes z). \qedhere
  \end{align*}
\end{proof}

\section{Examples}
 \label{sec:examples}
 \subsection{Classical graph} 
  Suppose \(\Qgh\) is a directed graph with vertex set \(V\) and edge set~\(E\), both finite. The state 
  \(\psi\) on \(\Cont(V)\) is the integration with respect to the uniform probability measure, which is a \(\sqrt{|V|}\)\nb-form, where \(|V|\) denotes the cardinality of \(V\). So, \(T^{*}=\flip\circ m^*\). 
  Recall that the quantum edge indicator and quantum edge correspondence computed in~\cite{BHINW2023a}*{Example 2.2} as
  \[
     \Eind[\Qgh]=\sum_{\substack{v_{0},v_{1}\in V,\\ v_{0}\leftarrow v_{1}}}p_{v_{0}}\otimes p_{v_{1}}, 
     \qquad 
     \Ecor[\Qgh]=\Cont(E).
  \]    
  For every positive integer \(k\), let \(E^{k}\) denotes the set of \(k\)\nb-length paths in \(\mathcal{G}\), that is, 
  \(E^{k}=\{(v_{0},v_{1},\cdots,v_{k})\in V^{k+1}\mid v_{i}\leftarrow v_{i+1}, i=0,\cdots k-1\}\). 
  Firstly, we compute \(2\)\nb-length quantum path indicator
  \begin{align*}
       \Eind[\Qgh]^{2}
     =\left(\Id_{\Cont(V)}\otimes (\Id_{\Cont(V)}\otimes A)T^*\right)\Eind[\Qgh]
     &=\sum_{\substack{v_{0},v_{1},w_{0},v_{2}\in V,\\ v_{0}\leftarrow v_{1}, w_{0}\leftarrow v_{2}}}p_{v_{0}}\otimes m(p_{v_{1}}\otimes p_{w_{0}})\otimes p_{v_{2}}\\
     &=\sum_{\substack{v_{0},v_{1},v_{2}\in V,\\ v_{0}\leftarrow v_{1} \leftarrow v_{2}}} p_{v_{0}}\otimes p_{v_{1}} \otimes p_{v_{2}}.
  \end{align*}   
  Thus \(\Eind[\Qgh]^{2}\in\Cont(V\times V\times V)\) is the indicator function of \(E^{2}\) and \(\Ecor[\Qgh]^{2}=\Cont(E^2)\). For \(k> 2\), by induction we can show that \(\Eind[\Qgh]^{k}\in\Cont(V^{k+1})\) is the indicator function of~\(E^{k}\) and \(\Ecor[\Qgh]^{k}=\Cont(E^{k})\).

\subsection{Complete quantum graph} Let \(\QSp\) be a finite quantum space with faithful \(\delta\)\nb-form \(\psi\). The complete quantum graph structure on~\((B,\psi)\) is given by quantum adjacency matrix \(A(x)=\delta^{2}\psi(x)1_{B}\) for all~\(x\in B\). We denote \(K_{\textup{dim}(B)}=\QghTr\). 

The quantum edge indicator~\(\Eind[K_{\textup{dim}(B)}]=1_{B}\otimes 1_{B}\) and quantum edge correspondence~\(\Ecor[K_{\textup{dim}(B)}]=B\otimes B\). Consequently, \(\Eind[K_{\textup{dim}(B)}]^{2}=(\Id_{B}\otimes T\otimes\Id_{B})(\Eind[K_{\textup{dim}(B)}]\otimes\Eind[K_{\textup{dim}(B)}])=1_{B}\otimes 1_{B}\otimes 1_{B}\) and \(\Ecor[K_{\textup{dim}(B)}]^{2}=B^{\otimes 3}\). Similarly, for every positive integer \(k\), \(\Eind[K_{\textup{dim}(B)}]^{k}=1_{B^{\otimes (k+1)}}\) and \(\Ecor[K_{\textup{dim}(B)}]^{k}=B^{\otimes (k+1)}\).

\subsection{Trivial quantum graph} Let \(\QSp\) be a finite quantum space with faithful \(\delta\)\nb-form \(\psi\). The trivial quantum graph structure on~\((B,\psi)\) is given by quantum adjacency matrix \(A=\Id_{B}\). We denote \(T_{\textup{dim}(B)}=\QghTr\). 

Let \(B=\oplus_{a=1}^{N}M_{n_{a}}(\C)\).  Using~\eqref{eq:adapted-mat-unit} and~\eqref{eq:adapted-mat-rel} we get the quantum edge indicator 
\[
   \Eind[T_{\textup{dim}(B)}]=\frac{1}{\delta^{2}}\sum_{a=1}^{N}\sum_{k,l=1}^{n_{a}}\psi(e_{kk}^{a})f_{kl}^{a}\otimes f_{lk}^{a}
\]   
and quantum edge correspondence 
\[
  \Ecor[T_{\textup{dim}(B)}]=\textup{span}\left\{\sum_{l=1}^{n_{a}}f_{rl}^{a}\otimes f_{ls}^{a}
  \mid 1\leq r,s\leq n_{a},\text{ } 1\leq a\leq N\right\}. 
\]
\cite{BHINW2023a}*{Proposition 4.4} shows that \(\frac{1}{\delta}m^{*}\colon B\to \Ecor[T_{\textup{dim}(B)}]\) is an isomorphism of \(\Cst\)\nb-correspondences. 

Let~\(k\) be a positive integer. Consider the external tensor product \(L^{2}(B,\psi)^{\otimes (k-1)}\otimes (B\otimes B)\) of the Hilbert
space~\(L^{2}(B,\psi)^{\otimes (k-1)}\) with the \(\Cst\)\nb-correspondence \(B\otimes B\) over \(B\), which is again a
\(\Cst\)\nb-correspondence over \(B\). The map \(f_{k}\colon L^{2}(B,\psi)^{\otimes (k-1)}\otimes (B \otimes B)\to B^{\otimes (k+1)}\), defined by \(f_{k}(x\otimes a\otimes b)=a\otimes x\otimes b\) for all~\(a,b\in B\) and \(x\in B^{\otimes (k-1)}\cong L^{2}(B,\psi)^{\otimes (k-1)}\), extends to an isomorphism of \(\Cst\)\nb-correspondences. 

\begin{proposition}
 For every positive integer \(k\) consider the \(\Cst\)\nb-subcorrespondence \( \oplus_{a=1}^{N}(L^{2}(M_{n_{a}}(\C),\psi_{a})^{\otimes (k-1)}\otimes M_{n_{a}}(\C))\subseteq L^{2}(B,\psi)^{\otimes (k-1)}\otimes B\) over \(B\). Then \(\frac{1}{\delta}f_{k}\circ (\Id_{B^{\otimes (k-1)}}\otimes m^{*})\colon \oplus_{a=1}^{N}(L^{2}(M_{n_{a}}(\C),\psi_{a})^{\otimes (k-1)}\otimes M_{n_{a}}(\C))\to\Ecor[T_{\textup{dim}(B)}]^{k}\) is an isomorphism of \(\Cst\)\nb-correspondences.
\end{proposition}
\begin{proof}
 Using Theorem~\ref{the:prop-p-ind-p-mod}(1), we compute 
 \begin{align*}
  \Eind[T_{\textup{dim}(B)}]^{2}
  &=(\Id_{B}\otimes T\otimes\Id_{B})
  (\Eind[T_{\textup{dim}(B)}]\otimes\Eind[T_{\textup{dim}(B)}])\\
  &=\frac{1}{\delta^{4}}\sum_{a,b=1}^{N}\sum_{k,l=1}^{n_{a}}\sum_{r,s=1}^{n_{b}}\psi(e_{kk}^{a})\psi(e_{rr}^{b}) 
  f_{kl}^{a}\otimes T(f_{lk}^{a}\otimes f_{rs}^{b})\otimes f_{sr}^{b}\\
  &=\frac{1}{\delta^{4}}\sum_{a,b=1}^{N}\sum_{k,l=1}^{n_{a}}\sum_{r,s=1}^{n_{b}}\psi(e_{kk}^{a})\psi(e_{rr}^{b}) 
  f_{kl}^{a}\otimes f_{rs}^{b}\sigma_{-i}^{\psi}(f_{lk}^{a})\otimes f_{sr}^{b}\\
&=\frac{1}{\delta^{4}}\sum_{a,b=1}^{N}\sum_{k,l=1}^{n_{a}}\sum_{r,s=1}^{n_{b}}\psi(e_{rr}^{b})\psi(e_{ll}^{a}) 
 f_{kl}^{a}\otimes f_{rs}^{b}f_{lk}^{a}\otimes f_{sr}^{b}\\
&=\frac{1}{\delta^{4}}\sum_{a=1}^{N}\sum_{k,l,r=1}^{n_{a}}\psi(e_{rr}^{a})
  f_{kl}^{a}\otimes f_{rk}^{a}\otimes f_{lr}^{a}. 
 \end{align*}
 Subsequently, 
 \[
   \Ecor[T_{\textup{dim}(B)}]^{2}
   =\textup{span}\left\{\sum_{l=1}^{n_{a}}
  f_{il}^{a}\otimes f_{rs}^{a}\otimes f_{lj}^{a}
  \mid 1\leq i,j,r,s\leq n_{a},\text{ } 1\leq a\leq N\right\}
 \]
 which is the image~\(\frac{1}{\delta}f_{2}\circ(\Id_{B}\otimes m^*)(\oplus_{a=1}^{N}(L^{2}(M_{n_{a}}(\C),\psi_{a})\otimes M_{n_{a}}(\C)))\). The proof for \(k> 2\) follows by induction.
\end{proof}

\end{document}